\documentclass[11pt]{amsart}
\usepackage{amsmath,amssymb,amsfonts,enumerate,amsthm, amscd,}
\usepackage{txfonts}
\newtheorem{theorem}{Theorem}[section]
\newtheorem{proposition}[theorem]{Proposition}

\newtheorem{corollary}[theorem]{Corollary}

\theoremstyle{definition}
\newtheorem{definition}[theorem]{Definition}
\newtheorem{example}[theorem]{Example}

\theoremstyle{remark}

\numberwithin{equation}{section}

\catcode`\ç=13
\defç{\c{c}}
\catcode`\é=13
\defé{\'e}
\catcode`\à=13
\defà{\`a}
\catcode`\è=13
\defè{\`e}
\catcode`\â=13
\defâ{\^a}
\catcode`\ù=13
\defù{\`u}
\catcode`\ê=13
\defê{\^e}
\catcode`\î=13
\defî{\^\i}
\catcode`\ô=13
\defô{\^o}

\def\id{{\rm id}}

\begin{document}

\title{Gaussian and Pr\"ufer conditions in bi-amalgamated algebras}

\author[Najib Mahdou]{Najib Mahdou}
\address{Department of Mathematics, Faculty of Science and Technology of Fez, Box 2202, University S. M.
Ben Abdellah Fez, Morocco}
 \email{mahdou@hotmail.com}

\author[Moutu Abdou Salam Moutui]{Moutu Abdou Salam Moutui}
\address{Department of Mathematics, College of Science,
King Faisal University, Saudi Arabia} \email{mmoutui@kfu.edu.sa}






\begin{abstract} Let $f: A\rightarrow B$ and $g: A\rightarrow C$ be two ring homomorphisms and let $J$ (resp., $J'$) be an ideal of $B$ (resp., $C$) such that $f^{-1}(J)=g^{-1}(J')$.
In this paper, we investigate the transfer of the notions of
Gaussian and Pr\"ufer properties to the bi-amalgamation of $A$ with
$(B,C)$ along $(J,J')$ with respect to $(f,g)$ (denoted by
$A\bowtie^{f,g}(J,J')),$ introduced and studied by Kabbaj, Louartiti
and Tamekkante in 2013. Our results recover well known results on
amalgamations in \cite{Finno} and generate new original examples of
rings satisfying these properties.
\smallskip
\end{abstract}

\maketitle

\section{Introduction} All rings considered in this paper are assumed to be commutative,
 and have identity element and all modules are unitary.\\

In 1932, Pr\"ufer
introduced and studied in \cite{P} integral domains in which every finitely generated ideal is invertible. In 1936, Krull \cite{Kru} named these
rings after H. Pr\"ufer and stated equivalent conditions that make a domain Pr\"ufer. Through the years, Pr\"ufer  domains acquired a great many
equivalent characterizations, each of which was extended to rings with zero-divisors in different ways. In their recent paper devoted to Gaussian
properties, Bazzoni and Glaz have proved that a Pr\"ufer ring satisfies any of the other four Pr\"ufer conditions if and only if its total ring of
quotients satisfies that same condition \cite[Theorems 3.3 \& 3.6 \& 3.7 \& 3.12]{BG2}.
In 1970, Koehler \cite{ko} studied associative rings for which every cyclic module is quasiprojective.
She noticed that any commutative ring satisfies this property.
 In \cite{bkm}, the authors examined the
transfer of the Pr\"ufer conditions and obtained further evidence
for the validity of Bazzoni-Glaz conjecture sustaining that "the
weak global dimension of a Gaussian ring is 0, 1, or $\infty$"
\cite{BG2}. Notice that both conjectures share the common context of
rings. Abuihlail, Jarrar and Kabbaj studied in \cite{abjk} the
multiplicative ideal structure of commutative rings in which every
finitely generated ideal is quasi-projective. They provide some
preliminaries quasi-projective modules over commutative rings and
they investigate the correlation with well-known Pr\"ufer
conditions; namely, they proved that this class of rings stands
strictly between the two classes of arithmetical rings and Gaussian
rings. Thereby, they generalized Osofsky's theorem on the weak
global dimension of arithmetical rings and partially resolve
Bazzoni-Glaz's related conjecture on Gaussian rings. They also
established an analogue of Bazzoni-Glaz results on the transfer of
Pr\"ufer conditions between a ring and its total ring of quotients.
In \cite{CJKM}, the authors studied the transfer of the notions of
local Pr\"ufer ring and total ring of quotients. They examined the
arithmetical, Gaussian, fqp conditions to amalgamated duplication
along an ideal. At this point, we make the following definition:
\begin{definition}  Let $R$ be a commutative ring.\\
\begin{enumerate}

\item $R$ is called an \emph{arithmetical ring} if the lattice formed by its ideals is distributive (see \cite{Fu}). \\
\item $R$ is called a \emph{Gaussian ring} if for every $f, g \in R[X]$, one has the content ideal equation $c(fg) = c(f)c(g)$ (see \cite{T}).
\item $R$ is called a \emph{Pr\"ufer ring} if every finitely generated regular ideal of $R$ is invertible (equivalently, every two-generated
    regular ideal is invertible), (See \cite{BS,Gr}).
\end{enumerate}
\end{definition}
In the domain context, all these forms coincide with the definition of a $Pr\ddot{u}fer$ domain. Glaz \cite{G3} provides
examples which show that all these notions are distinct in the context of arbitrary rings. The following diagram of implications
 summarizes the relations between them \cite{BG,BG2,G2,G3,LR,Lu,T}:
 \begin{center}
 Arithmetical $\Rightarrow$ Gaussian $\Rightarrow$ $Pr\ddot{u}fer$
\end{center}
and examples are given in \cite{G3} to show that, in general, the implications cannot be reversed.\\
In this paper, we investigate the transfer of Gaussian and Pr\"ufer
properties in bi-amalgamation of rings, introduced and studied by
Kabbaj, Louartiti and Tamekkante in \cite{KLT} and defined as follow
:
 Let $f: A\rightarrow B$ and $g: A\rightarrow C$ be two ring homomorphisms and let $J$ and $J'$ be two ideals
    of $B$ and $C$, respectively, such that $I_{o}:=f^{-1}(J)=g^{-1}(J')$. The \emph{bi-amalgamation}
    (or \emph{bi-amalgamated algebra}) of $A$ with $(B, C)$ along $(J, J')$ with respect to $(f,g)$ is the subring
    of $B\times C$ given by $$A\bowtie^{f,g}(J,J'):=\big\{(f(a)+j,g(a)+j') \mid a\in A, (j,j')\in J\times J'\big\}.$$
This construction was introduced in \cite{KLT} as a natural
generalization of duplications \cite{D,DF1} and amalgamations
\cite{DFF1,DFF2}. In \cite{KLT}, the authors provide original
examples of bi-amalgamations and, in particular, show that
Boisen-Sheldon's CPI-extensions \cite{BSh2}  can be viewed as
bi-amalgamations (notice that \cite[Example 2.7]{DFF1} shows that
CPI-extensions can be viewed as quotient rings of amalgamated
algebras). They also show how every bi-amalgamation can arise as a
natural pullback (or even as a conductor square) and then
characterize pullbacks that can arise as bi-amalgamations. Then, the
last two sections of \cite{KLT}  deal, respectively, with the
transfer of some basic ring theoretic properties to bi-amalgamations
and the study of their prime ideal structures.  All their results
recover known results on duplications and amalgamations. Recently in
\cite{KMM}, the authors established necessary and sufficient
conditions for a bi-amalgamation to inherit the arithmetical
property, with applications on the weak global dimension and
transfer of the semihereditary property. We will adopt the following notations:\\
For any $p\in Spec(A,I_{o})$ (resp., $\in Max(A,I_{o})$), consider
the multiplicative subsets
$$S_{p}:=f(A-p)+J\ \text{ and }\ S'_{p}:=g(A-p)+J'$$
of $B$ and $C$, respectively, and let
$$f_p: A_p\rightarrow B_{S_{p}}\ \text{ and }\ g_p: A_p\rightarrow C_{S'_{p}}$$
be the canonical ring homomorphisms induced by $f$ and $g$. One can
easily check that
$$f_p^{-1}(J_{S_{p}})=g_p^{-1}(J'_{S'_{p}})=(I_{o})_p.$$ Moreover, by \cite[Lemma 5.1]{KLT}, $P:=p\bowtie^{f,g}(J,J')$ is a prime (resp., maximal) ideal of $A\bowtie^{f,g}(J,J')$ and, by \cite[Proposition 5.7]{KLT}, we have

$$\big(A\bowtie^{f,g}(J,J')\big)_{P}\cong A_p\bowtie^{f_p,g_p}(J_{S_{p}},J'_{S'_{p}}).$$
For a ring $R$, we denote by $Jac(R)$, the Jacbson radical of $R$.
\section{Results}\label{sec:2}

\noindent Let $f: A\rightarrow B$ and $g: A\rightarrow C$ be two
ring homomorphisms and let $J$ and $J'$ be two ideals of $B$ and
$C$, respectively, such that $I_{o}:=f^{-1}(J)=g^{-1}(J')$. All
along this section, $A\bowtie^{f,g}(J,J')$ will denote the
bi-amalgamation of $A$ with $(B, C)$ along $(J, J')$ with respect to
$(f,g)$.\\
Our first result investigates the transfer of Gaussian and Pr\"ufer
properties in bi-amalgamated algebras in case $J\times J'$ contains
a regular element.

\begin{theorem}\label{thm1}
Assume $J\times J'$ is a regular ideal of
 $(f(A)+J)~\times (g(A)+~J')$. Then $A\bowtie^{f,g}(J,J')$ is Gaussian (resp.,
Pr\"ufer) if and only if $J=B$, $J'=C$ and $B$ and $C$ are Gaussian
(resp., Pr\"ufer).
\end{theorem}

\begin{proof}
Assume that $A\bowtie^{f,g}(J,J')$ is Gaussian (resp., Pr\"ufer). We
claim that $I_0=f^{-1}(J)=g^{-1}(J')=A$. Deny, suppose that there
exists a maximal ideal $m$ of $A$ such that $I_0\subseteq m$. From
\cite[Lemma 5.1]{KLT}, $M:=m\bowtie^{f,g}(J,J')$ is a maximal ideal
of $A\bowtie^{f,g}(J,J')$ and we have
$$\big(A\bowtie^{f,g}(J,J')\big)_{M}\cong
A_m\bowtie^{f_m,g_m}(J_{S_{m}},J'_{S'_{m}}):=D.$$ Let $(j,j')$ be a
regular element of $J\times J'$. It is easy to see that $j/1$
(resp., $j'/1$) is also a regular element of $B_{S_m}$ (resp.,
$C_{S_m}$). Using the fact $A\bowtie^{f,g}(J,J')$ is Gaussian
(resp., Pr\"ufer), then by \cite[Theorem 13]{Gr}, the ideals
$(j/1,0)D$ and $(j/1,j'/1)D$ are comparable. Since $0\neq j'/1,$
then necessarily $(j/1,0)D\subseteq (j/1,j'/1)D$. Thus, there exist
$\alpha\in A_m$, $\beta\in J_{S_{m}}$ and $\gamma\in J'_{S_{m}}$
such that
$(j/1,0)=(j/1,j'/1)(f_m(\alpha)+\beta,g_m(\alpha)+\gamma)$. Hence,
it follows that $f_m(\alpha)+\beta=1$ and $g_m(\alpha)+\gamma=0$.
Thus, $\alpha\in (I_0)_m$ and so $f_m(\alpha)\in J_{S_{m}}$ and
$1=f_m(\alpha)+\beta\in J_{S_{m}}.$ Therefore, $J_{S_{m}}=B_{S_m}$.
Then $(I_0)=A_m$, which is a contradiction since $I_0\subseteq m.$
Hence, $I_0=f^{-1}(J)=A$ and so $J=B$ and $J'=C$ and
$A\bowtie^{f,g}(J,J')=B\times C$ which is Gaussian (resp.,
Pr\"ufer). It is known that Gaussian (resp., Pr\"ufer) notion is
stable under finite products. It follows that $B$ and $C$ are
Gaussian (resp., Pr\"ufer). The converse is straightforward.
 \end{proof}

Recall that the amalgamation of $A$ with $B$ along $J$ with respect
to $f$ is given by
$$A\bowtie^{f} J:=\big\{(a,f(a)+j)\mid a\in A, j\in J\big\}.$$
Clearly, every amalgamation can be viewed as a special
bi-amalgamation, since $A\bowtie^fJ= A\bowtie^{{\id_{A}},
f}(f^{-1}(J),J)$.\\

The following result is an immediate consequence of Theorem
\ref{thm1} and recovers \cite[Theorem 3.1]{Finno}.
\begin{corollary}\label{Amalg}
Under the above notation, assume that $f^{-1}(J)\times J$ is a
regular ideal of $A\times f(A)+J$. Then $A\bowtie^{f}J$ is Gaussian
( resp., Pr\"ufer)  if and only if  $f^{-1}(J)=A$ and $J=B$ and both
$A$ and $B$ are Gaussian (resp., Pr\"ufer).
\end{corollary}

 Let $I$ be a \emph{proper} ideal of $A$. The (amalgamated) duplication of $A$ along
$I$ is a special amalgamation given by $$A\bowtie
I:=A\bowtie^{\id_{A}} I=\big\{(a,a+i)\mid a\in A, i\in I\big\}.$$
The next corollary is an immediate consequence of Corollary
\ref{Amalg} on the transfer of Gaussian and Pr\"ufer properties to
duplications and capitalizes \cite[Corollary 3.3]{Finno}.

\begin{corollary}Let $A$ be a ring and $I$ be a regular ideal of $A$. Then $A\bowtie I$  is Gaussian (resp., Pr\"ufer) if and only if $A$ is Gaussian (resp., Pr\"ufer) and $I=A$.

\end{corollary}

 The next result investigates when the bi-amalgamation is local Gaussian
in case $J\times J'$ is not a regular ideal. We recall an important
characterization of a local Gaussian ring $A$. Namely, for any two
elements $a$ and $b$ in the ring $A$, we have $(a,b)^2=(a^2)$ or
$(b^2)$; moreover if $ab=0$ and $(a,b)^2=(a^2)$, then $b^2=0$ (see
\cite[Theorem 2.2]{BG2}).

\begin{proposition}\label{gaus}
Assume that $(A,m)$ be a local ring and $J$ (resp., $J'$) be a
nonzero proper ideal of $B$ (resp., $C$) such that $J\times
J'\subseteq Jac(B\times C).$ Then the following statements hold:
\begin{enumerate}
\item If $A\bowtie^{f,g}(J,J')$ is Gaussian, then so are $f(A)+J$ and $g(A)+J'$.
\item If $A$, $f(A)+J$ and $g(A)+J'$ are Gaussian, $J^2=0$, $J'^2=0$, $\forall$ $a\in m,$
$f(a)J=f(a)^2J$ and $g(a)J'=g(a)^2J',$ then $A\bowtie^{f,g}(J,J')$
is Gaussian.
\item Assume that $A$ is Gaussian, $J^2=0$, $J'^2=0$ and $I_0$ is a prime ideal of
$A$. Then $A\bowtie^{f,g}(J,J')$ is Gaussian if and only if
$f(A)+J$, $g(A)+J'$ are Gaussian, $\forall$ $a\in m,$
$f(a)J=f(a)^2J$ and $g(a)J'=g(a)^2J'$.
\end{enumerate}
 \end{proposition}
\begin{proof}
Notice that from \cite[Proposition 5.4 (2)]{KLT}, $(A\bowtie^{f,g}(J,J'),m\bowtie^{f,g}(J,J'))$ is local since $(A,m)$ is local and $J\times J'\subseteq Jac(B\times C)$.\\
$(1)$  Since the Gaussian property is stable under factor rings
(here, $f(A)+J\simeq \frac{A\bowtie^{f,g}(J,J')}{0\times J'}$ and
$g(A)+J'\simeq \frac{A\bowtie^{f,g}(J,J')}{J\times 0}$ by
\cite[Proposition 4.1 (2)]{KLT}), then result is straightforward.\\
$(2)$ Assume that $A$, $f(A)+J$ and $g(A)+J'$ are Gaussian, $J^2=0$,
$J'^2=0$ and $\forall$ $a\in m,$ $f(a)J=f(a)^2J$ and
$g(a)J'=g(a)^2J'$. Our aim is to show that $A\bowtie^{f,g}(J,J')$ is
Gaussian. Let $(f(a)+i,g(a)+i')$ and $(f(b)+j,g(b)+j')\in
A\bowtie^{f,g}(J,J')$.
Two cases are possible:\\
Case 1: $a$ or $b\not \in m$. Assume without loss of  generality
that $a\not \in m$. Then $(f(a)+i,g(a)+i')\not \in
m\bowtie^{f,g}(J,J')$. So $(f(a)+i,g(a)+i')$  is invertible in
$A\bowtie^{f,g}(J,J')$. Therefore,
$((f(a)+i,g(a)+i'),(f(b)+j,g(b)+j'))^2=((f(a)+i,g(a)+i')^2)=A\bowtie^{f,g}(J,J')$.
Moreover, if
$((f(a)+i,g(a)+i'),(f(b)+j,g(b)+j'))^2=((f(a)+i,g(a)+i')^2)=A\bowtie^{f,g}(J,J')$
and $(f(a)+i,g(a)+i')(f(b)+j,g(b)+j')=(0,0)$, then it follows that $(f(b)+j,g(b)+j')=(0,0)$, making $(f(b)+j,g(b)+j')^2=(0,0),$ as desired.\\
Case 2: $a$ and $b\in m$. Using the fact that $A$ is local Gaussian,
then $(a,b)^2=(a^2)$ or $(b^2)$. We may assume that $(a,b)^2=(a^2)$.
So we have, $b^2=a^2x$ and $ab=a^2y$ for some $x, y\in A$. Moreover
$ab=0$ implies that $b^2=0$. So $f(b)^2=f(a)^2f(x)$,
$g(b)^2=g(a)^2g(x)$ and $f(a)f(b)=f(a)^2f(y),$
$g(a)g(b)=g(a)^2g(y)$. By assumption, $2f(b)j, f(b)i \in f(b)^2 J$
and $2f(a)if(x), f(a)j, 2f(a)if(y) \in f(a)^2 J.$ Therefore, there
exist $j_1,i_1,j_2,i_2,i_3\in J$ such that $2f(b)j=f(a)^2f(x)j_1,$
$2f(a)if(x)=f(a)^2i_1,$ $f(a)j=f(a)^2j_2$, $f(b)i=f(a)^2f(x)i_2,$
$2f(a)if(y)=f(a)^2i_3$ and similarly, there exist
$j_1',i_1',j_2',i_2',i_3'\in J'$ such that $2g(b)j'=g(a)^2g(x)j_1',$
$2g(a)i'g(x)=g(a)^2i_1',$ $g(a)j'=g(a)^2j_2'$,
$g(b)i'=g(a)^2g(x)i_2'$ and $2g(a)i'g(y)=g(a)^2i_3'$. In view of the
fact that $J^2=0$ and $J'^2=0$, one can easily check that
$(f(b)+j,g(b)+j')^2=(f(a)+i,g(a)+i')^2(f(x)+f(x)j_1-i_1,g(x)+g(x)j_1'-i_1')$
and
$(f(b)+j,g(b)+j')(f(a)+i,g(a)+i')=(f(a)+i,g(a)+i')^2(f(y)+f(x)i_2+j_2-i_3,g(y)-g(x)i'_2+j'_2-i'_3)$.
Consequently,
$((f(a)+i,g(a)+i'),(f(b)+j,g(b)+j'))^2=((f(a)+i,g(a)+i')^2)$.
Moreover, assume that $(f(a)+i,g(a)+i')(f(b)+j,g(b)+j')=(0,0)$.
Hence, $(f(a)+i)(f(b)+j)=0$ and $(g(a)+i')(g(b)+j')=0$.
 Since $((f(a)+i),(f(b)+j))^2=((f(a)+i)^2),$ $((g(a)+i'),(g(b)+j'))^2=((g(a)+i')^2),$ and $f(A)+J$ and $g(A)+J'$ are local Gaussian, then $(f(b)+j)^2=0$ and $(g(b)+j')^2=0$.
 Thus, $(f(b)+j,g(b)+j')^2=(0,0)$. Finally, $A\bowtie^{f,g}(J,J')$ is Gaussian, as desired.\\

$(3)$ If $A,$ $f(A)+J,$ $g(A)+J'$ are Gaussian, $J^2=0$, $\forall$
$a\in m,$ $f(a)J=f(a)^2J,$ $J'^2=0$ and $g(a)J'=g(a)^2J'$, then by
statement $(2)$ above, $A\bowtie^{f,g}(J,J')$ is Gaussian.
Conversely, assume that $A\bowtie^{f,g}(J,J')$ is Gaussian. Then by
 statement $(1)$ above, $f(A)+J$ and $g(A)+J'$ are Gaussian. Next,
 we show that $\forall$ $a\in m,$ $f(a)J=f(a)^2J$. It is clear that $f(a)^2J\subseteq
 f(a)J$. On the other hand, let $a\in
 m$ and $0\neq x\in J$. If $f(a)=0,$ then $f(a)J=f(a)^2J$, as
 desired. We may assume that $f(a)\neq 0$. Then obviously, $(0,0)\neq(f(a),g(a))$ and $(0,0)\neq(x,0)$ are elements of
 $A\bowtie^{f,g}(J,J')$. Using the fact $A\bowtie^{f,g}(J,J')$ is (local)
 Gaussian, then $((f(a),g(a)),(x,0))^2=((f(a),g(a))^2$ or
 $((x,0))^2$. Since $J^2=0$, say
 $((f(a),g(a)),(x,0))^2=((f(a),g(a))^2)$. If $(f(a),g(a))^2=(0,0)$,
 then it follows that $xf(a)=0$ and so $f(a)J\subseteq f(a)^2J$, as desired.
We may assume that $(f(a),g(a))^2\neq (0,0)$. And so there exists
$(f(b)+j,g(b)+j')\in  A\bowtie^{f,g}(J,J')$ such that
$(xf(a),0)=(f(a^2),g(a^2))(f(b)+j,g(b)+j')$. Therefore,

$$\left\{\begin{array}{cccc}
           xf(a)& = & (f(a^2)(f(b)+j)  & \quad (i)\\
             0 & = & (g(a^2b)+g(a^2)j') & \quad (ii)\\
         \end{array}\right.$$

From equation $(ii)$, it follows that $a^2b\in I_0$ which is prime
ideal of $A$. So $a^2\in I_0$ or $b\in I_0$. Two cases are possible:\\
Case 1: $a^2 \in I_0$. Then $a\in I_0$ and $f(a)\in J$. Therefore,
$f(a)J=f(a)^2J=0$ (as $J^2=0$).\\
Case 2:$b\in I_0$. Then $f(b)\in J$ and $f(b)+j\in J$. Consequently,
$xf(a)=(f(a^2)(f(b)+j)\in f(a)^2J$. Hence, $f(a)J\subseteq f(a)^2J$,
as desired. Next, it remains to show that $\forall$ $a\in m,$
$g(a)J'=g(a)^2J'$. Clearly, $g(a)^2J'\subseteq
 g(a)J'$. On the other hand, let $a\in m$ and $0\neq x'\in J'$. With similar
 argument as previously, it follows that $g(a)J'\subseteq g(a)^2J',$
 as desired.

\end{proof}

Proposition \ref{gaus} enriches the literature with new original examples of non-arithmetical Gaussian rings. Recall that for a ring $A$ and an $A-$module $E$,
the trivial ring extension of $A$ by $E$ (also called idealization of $E$ over $A$) is the ring $R:=A\propto E$ whose underlying group is $A\times E$ with multiplication given by $(a,e)$$(a',e')=(aa',ae'+ea')$.\\

 \begin{example}
 Let $(A,m):=(A_1\propto E_1 ,m_1\propto E_1)$ be the trivial ring extension of $A_1$ by $E_1$ which is a non-arithmetical Gaussian ring with $m_1^2=0$,
 (for instance $(A_1,m_1):=(\mathbb{Z}/4\mathbb{Z},2.\mathbb{Z}/4\mathbb{Z})$, $E_1$ be a nonzero $\frac{A_1}{m_1}-$vector space.
 By \cite[Theorem 2.1 (2) and (3)]{KMM}, $A$ is a non-arithmetical Gaussian ring, as $A_1$ is not a field).
 Let $B:=A\propto E$ be the trivial ring extension of $A$ by a nonzero $A/m-$vector space $E$. Consider $$\begin{array}{clcl}
  f: & A & \hookrightarrow & B \\
   & (a_1,e_1) & \hookrightarrow & f((a_1,e_1))=((a_1,e_1),0) \\
\end{array}$$ be an injective ring homomorphism and $J:=m\propto E=(m_1\propto E_1)\propto E$ be the maximal ideal of B. Let $C:=A_1$ and let $$\begin{array}{clcl}
  g: & A & \to & C \\
   & (a_1,e_1) & \to & g((a_1,e_1))=a_1 \\
\end{array}$$ be a surjective ring homomorphism and $J':=m_1$ be the maximal ideal of $C$. Clearly, $f^{-1}(J)=g^{-1}(J')=m_1\propto E_1$. Then : \\
$(1)$ $A\bowtie^{f,g}(J,J')$ is Gaussian.\\
$(2)$ $A\bowtie^{f,g}(J,J')$ is not arithmetical.\\
 \end{example}
\begin{proof}
$(1)$ One can verify that $J^2=0,$ $J'^2=0$, $f(a)J=f(a)^2J=0$, $g(a)J'=g(a)^2J'=0$ for all $a\in m$. Hence by using statement (2) of Proposition \ref{gaus}, it follows that $A\bowtie^{f,g}(J,J')$ is Gaussian.\\
$(2)$ By \cite[Theorem 2.1 (2)]{KMM}, $A\bowtie^{f,g}(J,J')$ is not
arithmetical since $f(A)+J=A\propto 0+m\propto E=A\propto E$ which
is not arithmetical (by \cite[Theorem 3.1 (3)]{bkm}, as $A$ is not a
field).
\end{proof}

Total rings of quotients are important source of Pr\"ufer rings.
Next, we study the transfer of this notion to bi-amalgamated
algebras, in case $J\times J'$ is not a regular ideal of
$(f(A)+J)\times (g(A)+J')$. For any ring $R$ and $J$ an ideal of
$R$, we denote by $Z(R)$ (resp., $Ann(J)$), the set of zero-divisor
elements of $R$ (resp., the annihilator of $J$).
\begin{proposition}\label{prufer}
Let $(A,m)$ be a local total ring of quotients, $f: A\rightarrow B,$
$g:A\rightarrow C$ be two ring homomorphisms, and let $J$ (resp.,
$J'$) be a nonzero proper ideal of $B$ (resp., $C$) such that
$f^{-1}(J)=g^{-1}(J')$, $J\times J'\subseteq Jac(B\times C)$. Assume
that $f$ is injective, $J^2=0$ and $J'^2=0$. Then
$A\bowtie^{f,g}(J,J')$ is a local total ring of quotients. In
particular, $A\bowtie^{f,g}(J,J')$ is Pr\"ufer.
\end{proposition}

\begin{proof}
Assume that $f$ is injective, $J^2=0$ and $J'^2=0$. By \cite[Proposition 5.4 (2)]{KLT}, $(A\bowtie^{f,g}(J,J'),m\bowtie^{f,g}(J,J'))$ is local since $(A,m)$ is local and $J\times J'\subseteq Jac(B\times C)$.\\
Our aim is to show that $A\bowtie^{f,g}(J,J')$ is a total ring of quotients, we have to prove that each element $(f(a)+i,g(a)+i')$ of $A\bowtie^{f,g}(J,J')$, is invertible or zero-divisor element.\\
Let $(f(a)+i,g(a)+i')$ be an element of $A\bowtie^{f,g}(J,J')$. If $a\not \in m$, then $a$ is invertible. And so $(f(a)+i,g(a)+i')\not\in m\bowtie^{f,g}(J,J')$ . Consequently, $(f(a)+i,g(a)+i')$ is invertible in $A\bowtie^{f,g}(J,J')$, as desired.\\
Now, we may assume that $a\in m$. If $a=0$, then
$(f(a)+i,g(a)+i')=(i,i')\in Z(A\bowtie^{f,g}(J,J')),$ since
$J^2=J'^2=0.$ We may assume $a\neq 0$. Since $A$ is local total ring of quotients, then there exists $0\neq b\in A$ such that $ab=0$. So $f(a)f(b)=0$ and $g(a)g(b)=0$. Two cases are then possible :\\
Case $1$: $f(b)\in Ann(J)$ and $g(b)\in Ann(J')$. Using the fact that $f$ is injective, then there exists $(0,0)\neq (f(b),g(b))\in A\bowtie^{f,g}(J,J')$ $/$ $(f(a)+i,g(a)+i')(f(b),g(b))=(0,0)$. Consequently, $(f(a)+i,g(a)+i')\in Z(A\bowtie^{f,g}(J,J'))$.\\
Case $2$: Assume that $f(b)\not\in Ann(J)$ or $g(b)\not \in Ann(J')$. Then there exists $0\neq k\in J$ or $0\neq k'\in J'$  such that $f(b)k\neq 0$ or $g(b)k'\neq 0$.
So, $(f(a)+i,g(a)+i')(f(b)k,0)=(0,0)$ or $(f(a)+i,g(a)+i')(0,g(b)k')=(0,0)$. Hence, $(f(a)+i,g(a)+i')\in Z(A\bowtie^{f,g}(J,J'))$.
 Thus, $A\bowtie^{f,g}(J,J')$ is a local total ring of quotients. In particular, $A\bowtie^{f,g}(J,J')$ is Pr\"ufer.\\

\end{proof}

Proposition \ref{prufer} enriches the current literature with new original examples of Pr\"ufer rings which are not Gaussian rings.\\

\begin{example}
Let $(A,m)$ be a non Gaussian local total ring of quotient (for
instance $(A,m):=(A_1\propto \frac{A_1}{m_1},m_1\propto
\frac{A_1}{m_1})$ with $(A_1,m_1)$ be a local ring that is not
Gaussian, by using \cite[Theorem 3.1 (1) and (2)]{bkm}). Let
$(B,N):=(A\propto E,m\propto E)$ be the trivial ring extension of
$A$ by the nonzero $\frac{A}{m}-$vector space $E$ and $C:=B\propto
E'$ be the trivial ring extension of $B$ by the nonzero
$\frac{B}{N}-$vector space $E'$. Consider
$$\begin{array}{clcl}
  f: & A & \hookrightarrow & B \\
   & (a_1,e_1) & \hookrightarrow & f((a_1,e_1))=((a_1,e_1),0) \\
\end{array}$$ be an injective ring homomorphism and $J:=0\propto E$ be a nonzero proper ideal of $B$ and let $$\begin{array}{clcl}
  g: & A & \hookrightarrow & C \\
   & (a_1,e_1) & \hookrightarrow & g((a_1,e_1))=((a_1,e_1),0),0) \\
\end{array}$$ be an injective ring homomorphism and let $J':=J\propto E'$ be a proper
ideal of $C$. Obviously, $f^{-1}(J)=g^{-1}(J')=0.$ Then :\\
$(1)$ $A\bowtie^{f,g}(J,J')$ is Pr\"ufer.\\
$(2)$ $A\bowtie^{f,g}(J,J')$ is not Gaussian.
\end{example}

\begin{proof}
$(1)$ We claim that $A\bowtie^{f,g}(J,J')$ is a local total ring of quotients. Indeed, by \cite[Proposition 5.3]{KLT}, $A\bowtie^{f,g}(J,J')$ is local since $A$ is local and $J\times J'\subseteq
Jac(B\times C)$. One can easily check that $J^2=0,$ $J'^2=0$. Hence, by using Proposition \ref{prufer}, it follows that $A\bowtie^{f,g}(J,J')$ is a total ring of quotients. Hence, $A\bowtie^{f,g}(J,J')$ is Pr\"ufer.\\
$(2)$ By $(1)$ of Proposition \ref{gaus}, $A\bowtie^{f,g}(J,J')$ is
not Gaussian since $f(A)+J=A\propto 0+0\propto E=A\propto E$ is not
Gaussian (By \cite[Theorem 3.1 (2)]{bkm}, since $A$ is not Gaussian,
as $A_1$ is not Gaussian).
\end{proof}

\bibliographystyle{amsplain}

\end{document}